\batchmode
\makeatletter
\def\input@path{{\string"/Users/paranoia/Documents/Research/mypapers/Vertex decomposable graphs and obstructions to shellability/\string"/}}
\makeatother
\documentclass[11pt,oneside]{amsart}
\usepackage[T1]{fontenc}
\usepackage[latin9]{inputenc}
\usepackage{verbatim}
\usepackage{amsthm}
\usepackage{graphicx}
\usepackage{amssymb}

\makeatletter
\numberwithin{equation}{section} 
\numberwithin{figure}{section} 
\theoremstyle{plain}
\theoremstyle{plain}
\newtheorem{thm}{Theorem}
  \theoremstyle{remark}
  \newtheorem{note}[thm]{Note}
  \theoremstyle{plain}
  \newtheorem{lem}[thm]{Lemma}
  \theoremstyle{definition}
  \newtheorem{condition}[thm]{Condition}
  \theoremstyle{plain}
  \newtheorem{cor}[thm]{Corollary}
  \theoremstyle{remark}
  \newtheorem{rem}[thm]{Remark}
  \newcounter{casectr}
  \newenvironment{caseenv}
  {\begin{list}{{\itshape\ Case} \arabic{casectr}.}{%
   \setlength{\leftmargin}{\labelwidth}
   \addtolength{\leftmargin}{\parskip}
   \setlength{\itemindent}{\listparindent}
   \setlength{\itemsep}{\medskipamount}
   \setlength{\topsep}{\itemsep}}
   \setcounter{casectr}{0}
   \usecounter{casectr}}
  {\end{list}}
 \theoremstyle{definition}
  \newtheorem{example}[thm]{Example}
  \theoremstyle{plain}
  \newtheorem{prop}[thm]{Proposition}

\usepackage{ae}
\usepackage{aecompl}

\usepackage{hyperref}

\makeatother

\begin{document}
%
{}

\global\long\def\normalin{\mathrel{\triangleleft}}

\global\long\def\innormal{\mathrel{\triangleright}}

\global\long\def\semidirect{\mathbin{\rtimes}}

\global\long\def\Stab{\operatorname{Stab}}

%
{}

\global\long\def\bdry{\partial}

\global\long\def\susp{\operatorname{susp}}

%
{}

\global\long\def\lrprod{\mathop{\check{\prod}}}

\global\long\def\lrtimes{\mathbin{\check{\times}}}

\global\long\def\urtimes{\mathbin{\hat{\times}}}

\global\long\def\urprod{\mathop{\hat{\prod}}}

\global\long\def\subsetdot{\mathrel{\subset\!\!\!\!{\cdot}\,}}

\global\long\def\dotsupset{\mathrel{\supset\!\!\!\!\!\cdot\,\,}}

\global\long\def\precdot{\mathrel{\prec\!\!\!\cdot\,}}

\global\long\def\dotsucc{\mathrel{\cdot\!\!\!\succ}}

\global\long\def\des{\operatorname{des}}

%
{}

\global\long\def\modreln{\mathrel{M}}

%
{}

\global\long\def\link{\operatorname{link}}

\global\long\def\freejoin{\mathbin{\circledast}}

\global\long\def\stellarsd{\operatorname{stellar}}

\global\long\def\conv{\operatorname{conv}}

\global\long\def\disjointunion{\mathbin{\dot{\cup}}}

%
{}

\global\long\def\cosetposet{\overline{\mathfrak{C}}}

\global\long\def\cosetlat{\mathfrak{C}}

\title{Vertex decomposable graphs \\
and obstructions to shellability}

\author{Russ Woodroofe}

\email{russw@math.wustl.edu}

\address{Department of Mathematics, Washington University in St.~Louis, St.~Louis,
Missouri, 63130}


\begin{abstract}
Inspired by several recent papers on the edge ideal of a graph $G$,
we study the equivalent notion of the independence complex of $G$.
Using the tool of vertex decomposability from geometric combinatorics,
we show that $5$-chordal graphs with no chordless $4$-cycles are
shellable and sequentially Cohen-Macaulay. We use this result to characterize
the obstructions to shellability in flag complexes, extending work
of Billera, Myers, and Wachs. We also show how vertex decomposability
may be used to show that certain graph constructions preserve shellability.
\end{abstract}
\maketitle

\section{Introduction}

Let $G=(V,E)$ be a graph with vertex set $V=\{x_{1},\dots,x_{n}\}$.
The \emph{independence complex} of $G$, denoted $I(G)$, is the simplicial
complex with vertex set $V$ and with faces the independent sets of
$G$. When it causes no confusion, we will say that $G$ satisfies
some property if its independence complex does. For example, we will
say that $G$ is shellable if $I(G)$ is shellable. The independence
complex has been previously studied in e.g. \cite{Aharoni/Berger/Meshulam:2005,Klivans:2007,Meshulam:2003}.

The Stanley-Reisner ring of $I(G)$ is \[
k[x_{1},\dots,x_{n}]/(x_{i}x_{j}\,:\, x_{i}x_{j}\in E).\]
The quotient in the above ring is also called the \emph{edge ideal}
of $G$ and has been an object of study in its own right \cite{Villarreal:2001}.
In particular, a recent series of papers \cite{Francisco/Ha:2008,Francisco/VanTuyl:2007,Herzog/Hibi/Zheng:2006,VanTuyl/Villarreal:2008}
has worked from the edge ideal to show that chordal graphs are sequentially
Cohen-Macaulay and shellable and that certain graph constructions
preserve shellability and/or being sequentially Cohen-Macaulay.

In this paper, we consider vertex decomposability in graphs. In Section
2, we recall the definition of a vertex decomposable simplicial complex
and show what this means for (the independence complexes of) graphs.
As an easy consequence we recover the result that chordal graphs are
shellable, hence sequentially Cohen-Macaulay. In Section 3, we give
a geometric proof that the only cyclic graphs which are vertex decomposable,
shellable and/or sequentially Cohen-Macaulay are $C_{3}$ and $C_{5}$.
In Section 4, we prove the main theorem of the paper:
\begin{thm}
\emph{(Main Theorem)} \label{thm:MainTheorem} If $G$ is a graph
with no chordless cycles of length other than $3$ or $5$, then $G$
is vertex decomposable (hence shellable and sequentially Cohen-Macaulay.)
\end{thm}
In Section 5, we reinterpret Theorem \ref{thm:MainTheorem} in terms
of obstructions to shellings, answering a question of Wachs. We also
give an application to domination numbers, in the style of \cite{Meshulam:2003}.
In Section 6, we examine several graph constructions that preserve
vertex decomposability. Finally, in Section 7 we close with some comments
on classes of sequentially Cohen-Macaulay graphs.
\begin{note}
Independence complexes have been studied more extensively in the combinatorics
literature as \emph{flag complexes} \cite[Chapter III.4 and references]{Stanley:1996}.
Many papers on flag complexes study them by considering the clique
complex. We notice that the clique complex of a graph $G$ is the
independence complex of the complement graph of $G$.
\end{note}

\subsection{Cohen-Macaulay complexes}

We review briefly the background definitions from geometric combinatorics
and graph theory.

A simplicial complex $\Delta$ is \emph{pure} if all of its facets
(maximal faces) are of the same dimension. A complex $\Delta$ is
\emph{shellable} if its facets fit together nicely. The precise definition
will not be important to us, but can be found, with much additional
background, in \cite[Lecture 3]{Wachs:2007}. The \emph{link} of a
face $F$ in $\Delta$ is\[
\link_{\Delta}F=\{G\,:\, G\cup F\mbox{ is a face in }\Delta,G\cap F=\emptyset\}.\]

Let $k$ be a field or the ring of integers. A complex $\Delta$ is
\emph{Cohen-Macaulay} over $k$ if $\tilde{H}_{i}(\link_{\Delta}F;k)=0$
for all faces $F$ and $i<\dim(\link_{\Delta}F)$. More intuitively,
a complex is Cohen-Macaulay if it has the homology of a bouquet of
top-dimensional spheres and if every link satisfies the same condition.
It is a well-known fact that any Cohen-Macaulay complex is pure. Any
pure, shellable complex is Cohen-Macaulay over any $k$. Our results
will be independent of the choice of $k$, and we henceforth drop
it from our notation.

Since simplicial complexes that are not pure are often interesting,
we study Stanley's extension \cite[Chapter III.2]{Stanley:1996} of
the definition of Cohen-Macaulay (and its relationship with shellability)
to arbitrary simplicial complexes. The \emph{pure $i$-skeleton} of
$\Delta$ is the complex generated by all the $i$-dimensional faces
of $\Delta$. A complex is \emph{sequentially Cohen-Macaulay} if the
pure $i$-skeleton is Cohen-Macaulay for all $i$. Any shellable complex
is sequentially Cohen-Macaulay. 

$\Delta$ is a Cohen-Macaulay complex if and only if the Stanley-Reisner
ring of $\Delta$ is a Cohen-Macaulay ring. There is also a ring-theoretic
notion of sequentially Cohen-Macaulay \cite[Definition III.2.9]{Stanley:1996}.
For more background, refer to \cite{Bjorner/Wachs/Welker:2009} and
\cite{Stanley:1996} for the combinatorial point of view or to \cite{Bruns/Herzog:1993}
for a more ring-theoretic approach.

\subsection{Chordless paths and cycles}

A \emph{chordless path} of length $n$ in a graph $G$ is a path $v_{1},v_{2},\dots,v_{n}$
in $G$ with no \emph{chord}, i.e. with no edge $v_{i}v_{j}$ with
$j\neq i+1$. Equivalently, the induced graph on $\{v_{1},\dots,v_{n}\}$
is the path on $n$ vertices. In a like manner, a \emph{chordless
cycle} of length $n$ is an induced $n$-cycle. 

A graph is \emph{$k$-chordal} if it has no chordless cycles of length
$>k$, and \emph{chordal} if it is $3$-chordal.

\section{Vertex decomposability and shedding vertices\label{sec:VertexDecGraphs}}

A simplicial complex $\Delta$ is recursively defined to be \emph{vertex
decomposable} if it is either a simplex or else has some vertex $v$
so that
\begin{enumerate}
\item both $\Delta\setminus v$ and $\link_{\Delta}v$ are vertex decomposable,
and
\item no face of $\link_{\Delta}v$ is a facet of $\Delta\setminus v$.
\end{enumerate}
A vertex $v$ which satisfies Condition (2) is called a \emph{shedding
vertex}. Vertex decompositions were introduced in the pure case by
Provan and Billera \cite{Provan/Billera:1980} and extended to non-pure
complexes by Björner and Wachs \cite[Section 11]{Bjorner/Wachs:1997}.

A vertex decomposable complex is shellable. One proof of this fact
is via the following lemma of independent interest:
\begin{lem}
\emph{\label{lem:SheddingVertexLemma}(Wachs \cite[Lemma 6]{Wachs:1999b})}
If $\Delta$ is a simplicial complex with shedding vertex $v$ and
if both $\Delta\setminus v$ and $\link_{\Delta}v$ are shellable,
then $\Delta$ is shellable.
\end{lem}
\noindent The shelling order in Lemma \ref{lem:SheddingVertexLemma}
is that of $\Delta\setminus v$, followed by the facets of $v*\link_{\Delta}v$
in the order of the shelling of $\link_{\Delta}v$.

To summarize, we have the chain of implications: \[
\mbox{vertex decomposable}\implies\mbox{shellable}\implies\mbox{sequentially Cohen-Macaulay.}\]
Both implications are known to be strict.

\medskip{}
The definition of vertex decomposable (and so of shedding vertex)
translates nicely to independence complexes. Let $N(v)$ denote the
\emph{open neighborhood} of $v$, that is, all vertices adjacent to
$v$. Let $N[v]$ denote the \emph{closed neighborhood} of $v$, which
is $N(v)$ together with $v$ itself, so that $N[v]=N(v)\cup\{v\}$.
\begin{lem}
\label{lem:VertDecompGraph}An independence complex $I(G)$ is vertex
decomposable if $G$ is a totally disconnected graph (with no edges)
or if
\begin{enumerate}
\item $G\setminus v$ and $G\setminus N[v]$ are both vertex decomposable,
and
\item no independent set in $G\setminus N[v]$ is a maximal independent
set in $G\setminus v$.
\end{enumerate}
\end{lem}
\begin{proof}
Translate the definitions!
\end{proof}
A shedding vertex of $G$ is any vertex which satisfies Condition
(2) of Lemma \ref{lem:VertDecompGraph}. A useful equivalent condition
for shedding vertices is:
\begin{condition}
\label{con:SheddingCondition}For every independent set $S$ contained
in $G\setminus N[v]$, there is some $x\in N(v)$ so that $S\cup\{x\}$
is independent.
\end{condition}
We make a first observation:
\begin{lem}
\label{lem:NeighborhoodContainmentShedding}If $N[v]\subseteq N[w]$
(so that in particular $v$ and $w$ are adjacent), then $w$ is a
shedding vertex for $G$.\end{lem}
\begin{proof}
Since every neighbor of $v$ is also a neighbor of $w$, there are
no edges from $v$ to any vertex of $G\setminus N[w]$. Thus, any
$v$ can be added to any independent set in $G\setminus N[w]$ while
preserving independence.
\end{proof}
Recall that a \emph{simplicial vertex} is a vertex $v$ such that
$N[v]$ is a clique. A well-known theorem of Dirac \cite[Theorem 6.3]{Alfonsin/Reed:2001}
says that every chordal graph has a simplicial vertex. Thus, we have:
\begin{cor}
\label{cor:SimplicialVertexIsShedding}~
\begin{enumerate}
\item Any neighbor of a simplicial vertex is a shedding vertex for $G$.
\item A chordal graph is vertex decomposable (hence shellable \cite[Theorem 1.2]{VanTuyl/Villarreal:2008}
and sequentially Cohen-Macaulay \cite[Theorem 1.2]{Francisco/VanTuyl:2007};
also closely related is \cite[Theorem 1.4]{Meshulam:2003}).
\end{enumerate}
\end{cor}
\begin{proof}
If $v$ is a simplicial vertex and $w$ is a neighbor of $v$, then
$N[v]\subseteq N[w]$, and Lemma \ref{lem:NeighborhoodContainmentShedding}
gives (1). For (2), the theorem of Dirac thus says that a chordal
graph has a shedding vertex $w$ if it is not totally disconnected.
Since every induced subgraph of a chordal graph is chordal, both $G\setminus w$
and $G\setminus N[w]$ are inductively vertex decomposable.\end{proof}
\begin{rem}
Vertices satisfying the condition of Lemma \ref{lem:NeighborhoodContainmentShedding}
have been studied before under the name dominant vertices, in the
context of so-called dismantlable graphs \cite{Ginsburg:1994,Boulet/Fieux/Jouve:2008UNP}.
However, dismantlability is a tool for understanding the homotopy
type of the clique complex of $G$, i.e., for understanding the independence
complex of the complement of $G$. Since $v$ and $w$ will not be
adjacent in the complement, there does not seem to be any direct interpretation
of dismantlability in terms of vertex decomposability.
\end{rem}

\begin{rem}
Anton Dochtermann and Alexander Engström also examined vertex decomposability
in graphs, independently and at about the same time \cite[Section\ 4]{Dochtermann/Engstrom:2009}
as I did. In particular, they prove Corollary \ref{cor:SimplicialVertexIsShedding},
and a special case of Proposition \ref{pro:PendantGen2}; they also
notice that the result of Billera and Myers discussed in Section \ref{sub:Obstructions}
is a special case of Corollary \ref{cor:SimplicialVertexIsShedding}.
\end{rem}

\section{Cyclic graphs}

Corollary \ref{cor:SimplicialVertexIsShedding} (2) states that if
$G$ has no chordless cycles of length greater than 3, then it is
vertex decomposable. Let $C_{n}$ be the cyclic graph on $n$ vertices.
We discuss a partial converse:
\begin{thm}
\label{thm:CyclicNotShellable}\emph{(Francisco/Van Tuyl \cite[Proposition 4.1]{Francisco/VanTuyl:2007})}
$C_{n}$ is vertex decomposable/shellable/sequentially Cohen-Macaulay
if and only if $n=3$ or $5$.
\end{thm}
Theorem \ref{thm:CyclicNotShellable} was proved with algebraic techniques
in \cite[Proposition 4.1]{Francisco/VanTuyl:2007}. We give a geometric
proof here. 

We start with a technical lemma:
\begin{lem}
\label{lem:OddTechnicalLemma}Let $n=2r+1$, $0<d<r$. Let the $d$-dimensional
complex $\Delta_{n}^{d}$ be the complex with vertex set $\mathbb{Z}/n\mathbb{Z}$
and with facets $F_{i}=\{i,i+2,\dots,i+2d\}$ for $i=1,\dots,n$.
Then $\Delta_{n}^{d}\cong S^{1}$. \end{lem}
\begin{proof}
Consider $\Delta_{n}^{d}$ for $d>1$. A facet $F_{i}$ has codimension
$1$ intersection with two other facets: $F_{i-2}$ and $F_{i+2}$.
Since $d>1$, and since all codimension $1$ faces of $F_{i}$ other
than $F_{i}\cap F_{i-2}$ and $F_{i}\cap F_{i+2}$ are {}``free''
(contained in a unique facet), we can collapse $F_{i}$ onto $F_{i}\cap F_{i-2}$
and $F_{i}\cap F_{i+2}$. More formally, the face $\{i,i+2d\}\subset F_{i}$
is free, so we can remove all faces containing $\{i,i+2d\}$ via an
elementary collapse \cite[Section 11.1]{Bjorner:1995}, which preserves
homotopy type. Every face $F$ not containing $\{i,i+2d\}$ is in
either $F_{i}\cap F_{i+2}$ (if $i\notin F$) or $F_{i}\cap F_{i-2}$
(if $i+2d\notin F)$.

Performing a similar collapse at each $F_{i}$ leaves us a simplicial
complex with facets $F_{i}\cap F_{i-2}$ for $i=1,\dots,n$. But $F_{i}\cap F_{i-2}=\{i,\dots,i+2d-2\}$,
and we see that we have collapsed $\Delta_{n}^{d}$ to $\Delta_{n}^{d-1}$.
Thus, $\Delta_{n}^{d}\cong\Delta_{n}^{d-1}$ when $d>1$.

Since $n$ is odd, repeatedly adding 2 to some $i\in\mathbb{Z}/n\mathbb{Z}$
will cover all vertices; hence $\Delta_{n}^{1}$ is the 1-complex
$C_{n}\cong S^{1}$.
\end{proof}
\begin{figure}[h]
\includegraphics{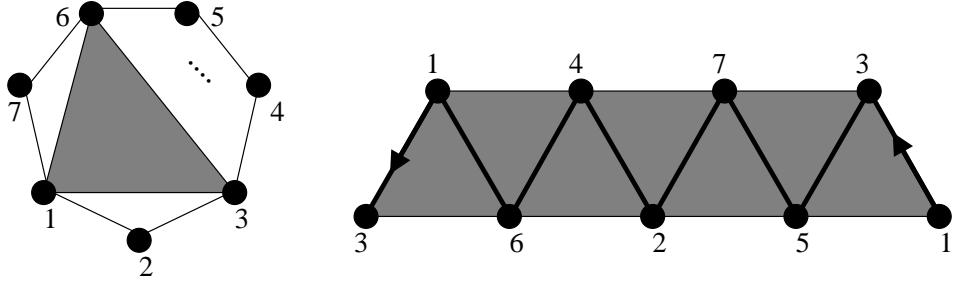}

\caption{$I(C_{7})$ is the Möbius strip. The dark line shows the subcomplex
$\Delta_{7}^{1}$. \label{fig:I(C_7)}}

\end{figure}

\begin{proof}
\emph{(Of Theorem \ref{thm:CyclicNotShellable})} The if direction
is easy: the independence complex of $C_{3}$ is three disconnected
vertices, while that of $C_{5}$ is $C_{5}$ as a 1-complex. Both
are clearly vertex decomposable.

In the other direction, we show that the pure $d$-skeleton is not
Cohen-Macaulay, where $d$ is the top dimension of the complex (i.e.,
$d=\dim I(C_{n})$). There are two cases, based on whether $n$ is
even or odd. (It may be helpful to look at Figures \ref{fig:I(C_7)}
and \ref{fig:I(C_6)} while reading the following.)
\begin{caseenv}
\item $n=2r$. Then the top-dimensional facets have dimension $r-1$, and
there are two of them: one with all even vertices, the other with
all odd vertices. As the pure $r-1$ skeleton is not even connected,
it is certainly not Cohen-Macaulay.

\begin{figure}[h]
\includegraphics{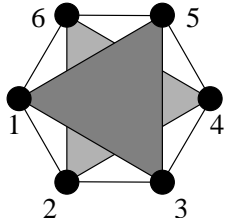}

\caption{The pure 2-skeleton of $I(C_{6})$ has two disconnected faces.\label{fig:I(C_6)}}

\end{figure}

\item $n=2r+1$, where $r\geq2$. Then the top-dimensional facets once more
have dimension $r-1$. All such facets are obtained by taking a sequence
of $r$ alternating vertices in $C_{n}$, with one skip of 2 vertices.
We see that the top-dimensional skeleton of $I(C_{n})$ is the complex
$\Delta_{n}^{r-1}$ discussed in Lemma \ref{lem:OddTechnicalLemma},
and so homotopic to $S^{1}$. Thus, the pure $r-1$ skeleton is Cohen-Macaulay
only when $r-1=1$, i.e., when $n=5$.
\end{caseenv}

Since every pure skeleton of a sequentially Cohen-Macaulay complex
is Cohen-Macaulay, we have shown that $C_{n}$ is not sequentially
Cohen-Macaulay (hence not shellable or vertex decomposable) for $n\neq3,5$.\end{proof}
\begin{example}
The pure 2-skeleton of $I(C_{6})$ consists of two disconnected triangles,
as shown in Figure \ref{fig:I(C_6)}, while $I(C_{7})$ is a (pure)
triangulation of the Möbius strip, as seen in Figure \ref{fig:I(C_7)}.
Lemma \ref{lem:OddTechnicalLemma} collapses $I(C_{7})$ to the cycle
$1,3,5,7,2,4,6$.\end{example}
\begin{rem}
We computed the homotopy type of the top-dimensional skeleton of $I(C_{n})$.
The homotopy type of the entire complex $I(C_{n})$ was calculated
by Kozlov \cite[Proposition 5.2]{Kozlov:1999}.
\end{rem}

\section{Proof of main theorem\label{sec:ProofOfMainTheorem}}

The previous two sections motivate the result of Theorem \ref{thm:MainTheorem}.
In this section, we will give a proof.

A \emph{simplicial $k$-path} in $G$ is a chordless path $v_{1},v_{2},\dots,v_{k}$
which cannot be extended on both ends to a chordless path $v_{0},v_{1},\dots,v_{k},v_{k+1}$
in $G$. Thus, a simplicial vertex is a simplicial $1$-path.

Chvátal, Rusu, and Sritharan \cite{Chvatal/Rusu/Sritharan:2002} proved
a nice generalization of Dirac's Theorem to $(k+2)$-chordal graphs
using simplicial $k$-paths. The following lemma of theirs will allow
us to use the $5$-chordal structure of $G$.
\begin{lem}
\emph{\label{lem:Chv=0000E1talRusuSritharan}(Chvátal, Rusu, and Sritharan
}\cite[Lemma 3]{Chvatal/Rusu/Sritharan:2002}\emph{) }If $G$ is a
$5$-chordal graph and $G$ contains a chordless $3$-path $P_{3}$,
then $G$ contains a simplicial $3$-path.\end{lem}
\begin{note}
From a geometric combinatorics point of view, it might make more sense
to count edge length and have the above definition be a simplicial
$(k-1)$-path, so that a simplicial vertex would be a simplicial $0$-path.
However, to avoid confusion, I have kept the original, more graph-theoretic
definition.
\end{note}
We also need to use the lack of chordless $4$-cycles:
\begin{lem}
\label{lem:C4FreeSimp3Path}Let $w_{1},v,w_{2}$ be a simplicial $3$-path
which is not a subgraph of any chordless $C_{4}$ in $G$. Then $v$
is a shedding vertex.\end{lem}
\begin{proof}
We first notice that, since there is no edge $w_{1}w_{2}$, that any
$z$ adjacent to both $w_{1}$ and $w_{2}$ must also be adjacent
to $v$. Otherwise, $w_{1},v,w_{2},z$ would be a chordless $4$-cycle. 

Suppose by contradiction that $v$ is not a shedding vertex. Then
by Lemma \ref{lem:VertDecompGraph} and Condition \ref{con:SheddingCondition},
there is an independent set in $G\setminus N[v]$ which contains a
vertex $z_{1}$ adjacent to $w_{1}$, and a vertex $z_{2}$ adjacent
to $w_{2}$. Since $z_{1},z_{2}\in G\setminus N[v]$, neither is adjacent
to $v$. No $z$ in $G\setminus N[v]$ is adjacent to both $w_{1}$
and $w_{2}$, so $z_{1}$ is not adjacent to $w_{2}$, and $z_{2}$
is not adjacent to $w_{1}$. Since $z_{1}$ and $z_{2}$ are in an
independent set, $z_{1}$ is not adjacent to $z_{2}$. 

Counting non-adjacent pairs of vertices, we have just shown that $z_{1},w_{1},v,w_{2},z_{2}$
is a chordless path, which contradicts the definition of simplicial
$3$-path. Thus $v$ is a shedding vertex, as desired.
\end{proof}

\begin{proof}
\emph{(Of Theorem \ref{thm:MainTheorem})} If $G$ is chordal, then
$G$ is vertex decomposable, as shown in Corollary \ref{cor:SimplicialVertexIsShedding}.
Otherwise, $G$ has some chordless $5$-cycle, hence a chordless $3$-path,
and by Lemma \ref{lem:Chv=0000E1talRusuSritharan} a simplicial $3$-path.
Lemma \ref{lem:C4FreeSimp3Path} shows that the middle vertex of any
simplicial $3$-path in $G$ is a shedding vertex, and so by induction
$G$ is vertex decomposable.
\end{proof}

\section{Applications}

\subsection{Obstructions\label{sub:Obstructions}}

An \emph{obstruction} \emph{to shellability} is a non-shellable complex,
all of whose proper subcomplexes are shellable. Thus, any non-shellable
complex must contain at least one obstruction to shellability, while
a shellable complex may or may not contain some obstructions to shellability
as proper subcomplexes.

The \emph{order complex} of a poset is the simplicial complex with
vertex set the elements of the poset and with face set the chains
of comparable elements. Thus, the order complex of $P$ is the independence
complex of the incomparability graph on $P$, which puts an edge between
two elements if they are incomparable. 

The study of obstructions to shellability was initiated by Billera
and Myers, with the following theorem:
\begin{thm}
\emph{\label{thm:Billera-and-Myers}(Billera and Myers \cite[Corollary 1]{Billera/Myers:1998})}
If $P$ is a non-shellable poset, then $P$ contains an induced subposet
isomorphic to the poset $D=\{\mbox{two disjoint edges}\}$.
\end{thm}
Equivalently, the unique obstruction to shellability in a poset is
$D$. We note that the incomparability graph of $D$ is $C_{4}$.
Gallai gave a forbidden subgraph characterization of incomparability
graphs of posets in \cite{Gallai:1967} (translated to English in
\cite[Chapter 3]{Alfonsin/Reed:2001}; a more accessible version of
the list is in \cite[Chapter 3.2]{Trotter:1992}). The forbidden subgraphs
include $C_{n}$ for $n\geq5$. Thus, Theorem \ref{thm:Billera-and-Myers}
follows from Corollary \ref{cor:SimplicialVertexIsShedding} (2).

Wachs studied obstructions to shellability further in \cite{Wachs:1999b},
where she asked about the obstructions to shellability in a flag complex.
Theorem \ref{thm:MainTheorem} gives a classification, which we summarize
in the following theorem:
\begin{thm}
\label{thm:FlagObstructions}The obstructions to shellability in flag
complexes are exactly the independence complexes of $C_{n}$, where
$n=4$ or $n\geq6$. \end{thm}
\begin{proof}
By Theorem \ref{thm:MainTheorem}, any non-shellable graph $G$ has
an induced subgraph (hence subcomplex) isomorphic to such a $C_{n}$.
In Theorem \ref{thm:CyclicNotShellable} we showed that such $C_{n}$
are not shellable, but any proper induced subgraph of $C_{n}$ is
chordal, hence shellable.
\end{proof}
A natural question suggested by Theorem \ref{thm:FlagObstructions}
is whether there is some similar characterization of obstructions
to shellability in non-flag complexes, where the minimal non-faces
form a hypergraph.\emph{ }One might be led to ask whether the hypergraph
of minimal non-faces is always cyclic in an obstruction to shellability.
However, examples studied by Wachs \cite{Wachs:1999b} show this is
not so, as follows.

Let $M_{n}$ be the simplicial complex with faces $\{1,2,3\}$, $\{2,3,4\}$,
$\dots$, $\{n-1,n,1\}$, $\{n,1,2\}$. In \cite[Lemma 5]{Wachs:1999b},
Wachs shows that $M_{5}$, $M_{6}$, and $M_{7}$ are obstructions
to shellability. 

Inspection verifies that $M_{7}$ is a flag complex, in fact that
$M_{7}=I(C_{7})$. The complexes $M_{5}$ and $M_{6}$ are not flag.
The minimal non-faces of $M_{5}$ are $\{1,2,4\},\{2,3,5,\},\{3,4,1\},\{4,5,2\},\{5,1,3\}$,
which is a cyclic hypergraph, insofar as there is an alternating sequence
of edges $e$ and vertices $v\in e$ which visits each edge and vertex
exactly once. However, the minimal non-faces of $M_{6}$ are $\{1,4\},\{2,5\},\{3,6\}$
and $\{1,3,5\},\{2,4,6\}$, as pictured in Figure \ref{fig:M6NonFaceHypergraph}.
I can see no natural generalization of cyclic graph which applies
directly to this hypergraph. Interestingly, however, there is an indirect
relationship: the edges are the facets of $I(C_{6})$.

\begin{figure}
\includegraphics{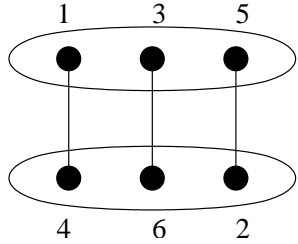}\caption{The minimal non-face hypergraph of $M_{6}$.\label{fig:M6NonFaceHypergraph}}

\end{figure}

\subsection{Domination numbers}

A set $S\subseteq V$ is a \emph{dominating set} if $\bigcup_{s\in S}N[s]=V$.
The \emph{dominating number} of $G$, denoted $\gamma(G)$, is the
minimum cardinality of a dominating set. Meshulam showed \cite[Theorem 1.2 (iii)]{Meshulam:2003}
that the homology of $I(G)$ vanishes below dimension $\gamma(G)-1$
when $G$ is a chordal graph.

We generalize this result in two respects. Let $i(G)$ be the \emph{independent
domination number}, that is, the minimum cardinality of a maximal
independent set. Any maximal independent set is a minimal dominating
set, so $\gamma(G)\leq i(G)$. 

Since a sequentially Cohen-Macaulay complex has homology vanishing
below the dimension of the smallest facet, and since the smallest
facet of $I(G)$ has cardinality $i(G)$, an immediate consequence
is the following:
\begin{cor}
If $G$ is any sequentially Cohen-Macaulay graph (over $k$), then
$\tilde{H}_{j}(I(G);k)=0$ for any $j<i(G)-1$.
\end{cor}
In particular, we recover the result \cite[Theorem 1.2 (iii)]{Meshulam:2003}
for chordal graphs and $\gamma(G)$ and extend it to a larger class
of graphs and a larger graph invariant.

\section{Graph constructions}

We now give examples of how shedding vertices can be used to show
that certain graph constructions respect shellability. 
\begin{lem}
\label{lem:DisjointUnionIsJoin}If $G=G_{1}\disjointunion G_{2}$,
then $I(G)=I(G_{1})*I(G_{2})$, the join of simplicial complexes.
Hence $G$ is vertex decomposable, shellable and/or sequentially Cohen-Macaulay
if and only if $G_{1}$ and $G_{2}$ are.\end{lem}
\begin{proof}
It is obvious from the definition that $I(G)$ is the given join and
that the join is vertex decomposable if and only if both $I(G_{1})$
and $I(G_{2})$ are. That the join of two complexes is shellable or
sequentially Cohen-Macaulay if and only if both complexes are is well
known and can be found for example in \cite{Wachs:2007}. (Part of
Lemma \ref{lem:DisjointUnionIsJoin} can be found as \cite[Lemma 2.4]{VanTuyl/Villarreal:2008}.)\end{proof}
\begin{example}
Adding a single vertex to $G$ via disjoint union forms a cone over
$I(G)$. Adding on the graph consisting of two vertices connected
by an edge via disjoint union corresponds to taking the suspension
of $I(G)$. Thus, for example, the union of $n$ disjoint edges is
homotopic to $S^{n-1}$.
\end{example}
Francisco and Hà \cite{Francisco/Ha:2008}, following Villarreal \cite[Theorem 2.2]{Villarreal:1990},
define a \emph{whisker} in a graph as a vertex of degree 1. A similar
idea seems to be studied in the wider graph theory literature under
the name of \emph{pendant}. We will prefer the latter term here. In
\cite{Francisco/Ha:2008} and \cite{VanTuyl/Villarreal:2008}, it
is shown that, speaking broadly, adding pendants to graphs has good
properties for maintaining shellability and the sequentially Cohen-Macaulay
property. Their construction essentially works because adding a pendant
adds a simplicial vertex. We give an obvious generalization:
\begin{prop}
\label{pro:PendantGen2}Let $G_{0}$ be a graph with a complete subgraph
$K$, and let $G$ be obtained from $G_{0}$ by adding a new vertex
$v$ with edges to all vertices of $K$. (That is, let $G$ be obtained
from $G_{0}$ by {}``starring $K$''.) Then any element of $K$
is a shedding vertex in $G$; conversely, $G$ is shellable (sequentially
Cohen-Macaulay) only if $G_{0}\setminus K$ is.\end{prop}
\begin{proof}
Since $N[v]=K\cup\{v\}$, we have that $v$ is a simplicial vertex;
hence any neighbor is a shedding vertex (Corollary \ref{cor:SimplicialVertexIsShedding}).
For the converse statement, we recall that links in a shellable/sequentially
Cohen-Macaulay complex have the same property and notice that $\link_{I(G)}v=G\setminus N[v]=G_{0}\setminus K$.
\end{proof}
The {}``clique-starring'' construction described in Proposition
\ref{pro:PendantGen2} adds a pendant when $\vert K\vert=1$. Whatever
the size of $K$, the construction adds a simplicial vertex to $G$.
We now consider a construction analogous to a pendant which adds a
$3$-simplicial path. 
\begin{prop}
Let $G_{0}$ be a graph with a complete subgraph $K$, and let $K_{1},K_{2}$
be disjoint subgraphs of $K$. Let $G$ be obtained from $G_{0}$
by adding new vertices $w_{1}$, $w_{2}$, and $v$, with $w_{1}$
adjacent to all vertices of $K_{1}$, $w_{2}$ adjacent to all vertices
of $K_{2}$, and $v$ adjacent to $w_{1}$ and $w_{2}$. Then $v$
is a shedding vertex of $G$. Conversely, $G$ is shellable (sequentially
Cohen-Macaulay) only if $G_{0}$ is.\end{prop}
\begin{proof}
By definition, the path $w_{1},v,w_{2}$ is $3$-simplicial, while
$v$ is in no chordless $4$-cycles because $K_{1}$ and $K_{2}$
are disjoint. Lemma \ref{lem:C4FreeSimp3Path} shows that $v$ is
a shedding vertex. Conversely, $\link_{I(G)}v=G_{0}$, and any link
in a shellable/sequentially Cohen-Macaulay complex has the same property. 

We notice in passing that $G\setminus v=G_{0}\cup\{w_{1},w_{2}\}$
is formed from $G_{0}$ by performing the construction of Proposition
\ref{pro:PendantGen2} on $K_{1}$ and $K_{2}$.
\end{proof}
Another graph construction is that of twinning. If two vertices $v$
and $w$ have the same neighbors, i.e., if $N(w)=N(v)$, then we say
$v$ and $w$ are \emph{true twins} if there is an edge $vw$, and
\emph{false twins} otherwise. There are corresponding graph constructions:
add a new vertex $w$ to $G$, together with edges to all neighbors
of some $v$. \label{sec:DistanceHered}The family of \emph{distance
hereditary graphs} can be defined as the graphs that can be built
from a single vertex by adding pendants, true twins, and false twins
\cite[Chapter 11.6]{Brandstadt/Le/Spinrad:1999}; twins are also useful
in proofs of the Perfect Graph Theorem \cite[Chapter 5.5]{Diestel:2005}.
\begin{prop}
If $v$ and $w$ are true twins, then $v$ and $w$ are shedding vertices.\end{prop}
\begin{proof}
We note that $N[v]=N[w]$, and thus by Lemma \ref{lem:NeighborhoodContainmentShedding}
they are both shedding vertices.
\end{proof}
False twins are never shedding vertices, since any maximal independent
set including $w$ in $G\setminus N[v]$ is also maximal in $G\setminus v$. 
\begin{example}
The 4-cycle is obtained by adding a false twin of the middle vertex
in a 3-path. Thus, adding a false twin to a sequentially Cohen-Macaulay
graph can result in a non-sequentially Cohen-Macaulay graph.

\medskip{}

One more family of graph operations which has frequently been studied
is that of graph products. There are a large number of such operations,
based on different rules for putting edges on the Cartesian product
of the vertex set. We cannot examine all of them, but note that the
commonly considered operations of direct product and Cartesian product
of graphs do not respect shellability or the sequential Cohen-Macaulay
property, for the Cartesian product of two edges (shellable) is a
4-cycle (not sequentially Cohen-Macaulay), while the direct product
of an edge with a 3-cycle (both shellable) is the complete bipartite
graph $K_{3,3}$, which \cite[Corollary 3.11]{VanTuyl/Villarreal:2008}
shows is not sequentially Cohen-Macaulay.
\end{example}

\section{A comment on perfect graphs}

Herzog, Hibi, and Zheng \cite{Herzog/Hibi/Zheng:2006} point out that
classifying sequentially Cohen-Macaulay graphs is likely an intractable
problem. We recall their argument. If $\Delta$ is a simplicial complex,
then the order complex of the face lattice of $\Delta$ is a flag
complex, and it is sequentially Cohen-Macaulay if and only if $\Delta$
is. (The order complex of the face lattice is the barycentric subdivision
of $\Delta$.) Herzog, Hibi and Zheng conclude that characterizing
sequentially Cohen-Macaulay graphs is as difficult as characterizing
all sequentially Cohen-Macaulay complexes. The closely related property
of shellability is likely of a similar difficulty.

As we have seen, however, there are families of graphs in which classifying
the sequentially Cohen-Macaulay members of the family is possible.
That all chordal graphs are sequentially Cohen-Macaulay (Corollary
\ref{cor:SimplicialVertexIsShedding}) is an example of this type
of classification, as is the recursive characterization of sequentially
Cohen-Macaulay bipartite graphs in \cite[Corollary 3.11]{VanTuyl/Villarreal:2008}.
Other families of graphs may also have interesting answers. 

We notice that the argument of Herzog, Hibi, and Zheng can help indicate
the families of graphs in which we can hope for such a classification.
For example, a \emph{perfect graph }is one where every induced subgraph
has chromatic number equal to the size of its largest clique. The
Strong Perfect Graph Theorem says that a graph $G$ is perfect if
and only if there are no chordless odd cycles of length $\geq5$ in
either $G$ or its complement. Another fundamental result is that
the complement of a perfect graph is also perfect. See \cite{Alfonsin/Reed:2001}
for more information and references about perfect graphs. Both chordal
graphs and bipartite graphs are perfect, and characterizing the shellability
and/or sequential Cohen-Macaulay connectivity of their common super-family
would seem like a reasonable aim.

Unfortunately for this aim, poset (in)comparability graphs are perfect,
as can be proved either by the direct argument of coloring elements
by their rank, or else from Gallai's previously mentioned characterization
of poset comparability graphs, which lack odd cycles of length $\geq5$
\cite{Gallai:1967}. Moreover, a complex is sequentially Cohen-Macaulay
if and only if its face poset is sequentially Cohen-Macaulay. Thus,
characterizing the sequentially Cohen-Macaulay perfect graphs is at
least as hard as characterizing which complexes in general are sequentially
Cohen-Macaulay. 

Considering the intersection of a graph family $\mathcal{F}$  with
the family of poset incomparability graphs is a recommended exercise
before looking for shellings of graphs in $\mathcal{F}$.

\section*{Acknowledgements}

I would like to thank Chris Francisco and John Shareshian for helpful
references and stimulating discussions.

\bibliographystyle{hamsplain}
\bibliography{5_Users_paranoia_Documents_Research_Master}

\end{document}